\documentclass{amsart}

\usepackage{amsmath}
\usepackage{amssymb,amsfonts,amsthm}
\usepackage{cite}
\usepackage{graphicx}
\usepackage[colorlinks=true, allcolors=blue]{hyperref}
\usepackage[all]{xy}
\usepackage{color,xcolor}

\usepackage{mathrsfs}
\usepackage[all]{xy}
\usepackage{mathabx} 
\usepackage{mathtools}
\usepackage{mathbbol}

\usepackage{wasysym}
\usepackage{cancel,ulem}

\usepackage{mathtools}
\usepackage{mathbbol}
\usepackage{tikz}

\usepackage{ifthen}

\newtheorem{theorem}{Theorem}[section]
\newtheorem{lemma}[theorem]{Lemma}
\newtheorem{proposition}[theorem]{Proposition}
\newtheorem{corollary}[theorem]{Corollary}

\theoremstyle{definition}

\theoremstyle{remark}
\newtheorem{remark}[theorem]{Remark}

\newtheorem{fact}{Fact}

\def\Z{\mathbb Z}

\def\F{\mathbb F}
\def\N{\mathbb N}

\def\Q{\mathbb Q}
\def\G{\mathcal G}
\def\C{\mathcal C}
\def\z{\mathcal Z}
\def\E{\mathcal E}

\def\mcg{\mathrm{MCG}}

\def\sl{\mathrm{SL}}

\title{Genericity of hyperbolic 3-manifolds via Dehn surgery}

\author{Renxing Wan}
\address{School of Mathematical Sciences,  Key Laboratory of MEA (Ministry of Education) \& Shanghai Key Laboratory of PMMP,  East China Normal University, Shanghai 200241, China P. R.}
\email{rxwan@math.ecnu.edu.cn}

\author{Yanqing Zou}
\address{School of Mathematical Sciences,  Key Laboratory of MEA (Ministry of Education) \& Shanghai Key Laboratory of PMMP,  East China Normal University, Shanghai 200241, China P. R.}
\email{yqzou@math.ecnu.edu.cn}


\begin{document}

\begin{abstract}
     A significant result by Lickorish and Wallace shows that every closed, orientable 3-manifold can be obtained from a Dehn surgery on one link in 3-sphere. As links and Dehn surgeries vary vastly in the universe, a question arises: how can we describe their properties in vague?

     We introduce a counting model on links and Dehn surgeries, and prove that under this model, (1) a randon link is hyperbolic; (2) a random 3-manifold is hyperbolic. 
    
\end{abstract}
\maketitle

\section{Introduction}

In the recent years, there has been an increasing interest in using probabilistic methods in low-dimensional geometry and topology. A number of models for 3 manifolds and links appeared, and were used to study their topological properties and various invariants. 

As a pioneering work, in \cite{DT06}, Dunfield and Thurston introduced a random model of 3-manifolds by using random walks on the mapping class group of a surface, and a theory of random 3-manifolds has started. Actually they considered random Heegaard splittings by gluing a pair of handlebodies by the result of a random walk in the mapping class group. By utilizing this model, they showed that a generic closed 3-manifold is a $\Q$-homology sphere but not a $\Z$-homology sphere (see \cite[Corollary 8.5]{DT06}).

Later random Heegaard splittings are studied extensively by Maher in \cite{Mah10}. In particular, he showed that a 3-manifold obtained by a random Heegaard splitting is hyperbolic with asymptotic probability 1. More recently, Han, Yang and the second author \cite{HYZ25} introduced a new counting model for random 3-manifolds via counting orbits in proper actions of mapping class groups on Teichm\"{u}ller spaces and obtained the genericity of hyperbolic closed 3-manifolds. 

In contrary, by a classical result of Lickorish and Wallace, every closed orientable 3-manifold can be obtained from a Dehn surgery along a link. This motivates us to consider another random model for closed 3-manifolds, namely random Dehn surgeries performed on random links. In particular, we reformulate such random Dehn surgeries in terms of random group elements in a finite direct product of $\sl(2,\Z)$.

Since every link in $S^3$ is the closure of some braid \cite{Ale23},  the typical way to build random links is through random walks on the braid group\cite{Ma14}. 
Recently,  counting models on the braid group or mapping class groups are more attractive due to its close ties with dynamics of these two groups and 3-manifolds, see \cite{Cho25a,Cho25b,HYZ25}. Therefore, we study the following counting model in braid groups and generate a random link.

Let $G$ be a finitely generated group with a finite generating set $F$. For a positive number $R>0$, we denote $B_F(R)$ as the collection of group elements whose word norm with respect to $F$ is at most $R$. Then the \textit{counting measure} (considered in this paper) on $G$ is the following density function $$\delta: 2^G\to [0,1],\quad \delta(A):=\lim_{R\to \infty}\frac{\sharp(B_F(R)\cap A)}{\sharp B_F(R)}.$$ If $\delta(A)=1$ (resp. $\delta(A)=0$), then we say $A$ is \textit{generic} (resp. \textit{negligible}) in $G$. If $\delta(A)>0$, then we say $A$ has a \textit{positive density} in $G$.

With the above counting measure of the braid group and the mapping class group $\sl(2,\Z)$ of a torus, we  study the typical behavior of links and  the effect of random Dehn fillings on a link complement and thus the typical behavior of closed 3-manifolds. By utilizing these counting models, we obtain some classical generic results and give some new results in 3-manifolds.

\begin{theorem}
    By considering the counting measures on $B_n(n>3)$ and $\sl(2,\Z)$, we have
    \begin{enumerate}
        \item A generic braid in $B_n$ gives a hyperbolic link (see Theorem \ref{Thm: HypLink}).
        \item The set $\mathcal C_n\subset B_n$ consisting of braids whose closures are knots has a positive density in $B_n$. Moreover, a generic braid in $\C_n$ gives a hyperbolic knot (see Lemma \ref{Lem: PosDen} + Theorem \ref{Thm: HypKnot}).
        \item A generic closed orientable 3-manifold is a $\Q$-homology sphere but not a $\Z$-homology sphere (see Theorem \ref{Thm: GenQHomSph} + \ref{Thm: GenNonZHomSph}).
        \item A generic closed orientable 3-manifold is hyperbolic (see Theorem \ref{Thm: GenHyp3Mfd}).
    \end{enumerate}
\end{theorem}

{\bf Acknowledgment.} This work was partially supported by the National Key R \& D program of China (2025YFA1017500), NSFC 12131009 \& 12471065, and in part by Science and Technology Commission of Shanghai Municipality (No. 22DZ2229014).


\section{Preliminaries}
\subsection{Dehn filling and Dehn surgery}

Let $M$ be a 3-manifold with a torus boundary $T$. Fix a meridian $m$ and a longitude $l$ of $T$. A $(p,q)$-Dehn filling of $M$ along $T$ is the operation of gluing
a solid torus to $M$ by killing the curve $pm+ql$. Precisely speaking, let $M_0$ be a solid torus with the torus boundary $T_0$. Fix a meridian $m_0$ and a longitude $l_0$ of $T_0$. For each pair of coprime integers $(p,q)$, there exists a pair of coprime integers $(r,s)$ such that $rq-sp=1$. Then $M_0$ kills the curve $pm+ql$ means that the gluing map $\phi: T_0\to T$ is given by $\phi(m_0)=pm+ql,\quad  \phi(l_0)=rm+sl$. This gives a matrix representative $\begin{pmatrix}
    r & s\\ p & q
\end{pmatrix}$ of $\phi$ in $\sl(2,\Z)$. On the other hand, given a matrix $\phi=\begin{pmatrix}
    r & s\\ p & q
\end{pmatrix}\in \sl(2,\Z)$, it also gives a $(p,q)$-Dehn filling of $M$ along $T$ by taking $\phi$ as the gluing map.

Since a Dehn filling is essentially determined by the killed slope, a $(p,q)$-Dehn filling can be also stated as a $p/q$-Dehn filling.

Let $\xi: \sl(2,\Z)\to \Q\cup\{\infty\}$ be a map defined by $\xi(\begin{pmatrix}
    r & s\\ p & q
\end{pmatrix})=p/q$. Then each $\phi\in \sl(2,\Z)$ determines a $\xi(\phi)$-Dehn filling.

Let $L\subset S^3$ be a link with $k$ components. A Dehn surgery on $L$ is a Dehn filling of the complement of $L$. In other words, a Dehn surgery on $L$ is a $(p_1/q_1,\cdots,p_k/q_k)$-Dehn filling of $M=S^3\setminus N(L)$ along $\partial M$.


The following is a fundamental theorem in 3-dimensional topology. We refer the readers to \cite[Theorem 11.3.15]{Mar16} for its proof.
\begin{theorem}[Lickorish-Wallace]\label{Thm: LW}
    Every closed orientable 3-manifold can be described via an integral Dehn surgery along a link $L\subset S^3$.
\end{theorem}

\subsection{Genericity of Anosov matrices}
Let $G$ be a finitely generated group with a finite generating set $F$. We denote $B_F(R)$ as the collection of group elements whose word norm with respect to $F$ is at most $R$. 
\begin{theorem}\cite[Theorem 1.1]{GTT18}\label{Thm: GTT}
    Let $G$ be a hyperbolic group with a nonelementary action by isometries on a separable, hyperbolic metric space $X$. Then $X$-loxodromics are generic on $G$, i.e. $$\lim_{R\to \infty}\frac{\sharp\{g\in B_F(R): g \text{ is loxodromic on } X\}}{\sharp B_F(R)}= 1$$ for any finite generating set $F$ of $G$.
\end{theorem}

It is well-known (cf. \cite[\textsection 13.1]{FM12}) that any $\phi\in \mcg(T^2)\cong \sl(2,\Z)$ can be classified into the following three types: periodic, reducible and Anosov. Moreover, a matrix representative of $\phi\in \sl(2,\Z)$ is periodic (or reducible, or Anosov) if and only if its trace has an absolute value $<2$ (or $=2$, or $>2$).

\begin{fact}
    \begin{enumerate}
        \item \cite[\textsection 4.4]{Loh17} $\sl(2,\Z)$ is virtually free, and thus word-hyperbolic.
        \item  There is a non-elementary isometric group action of $\sl(2,\Z)$ on the upper half plane $\mathbb H^2$ given by $$\begin{pmatrix}
    r & s\\ p & q
\end{pmatrix}\cdot z\mapsto \frac{r\cdot z+s}{p\cdot z+q}.$$ The set of $\mathbb H^2$-loxodromics coincides with the set of Anosov matrices.
    \end{enumerate}
\end{fact}

As a corollary of Theorem \ref{Thm: GTT}, we have
\begin{corollary}\label{Cor: AnoIsGen}
    The set of Anosov matrices in $\sl(2,\Z)$ is generic.
\end{corollary}

\begin{lemma}\label{Lem: p=0}
    Let $\phi=\begin{pmatrix}
    r & s\\ p & q
\end{pmatrix}\in \sl(2,\Z)$. If $p=0$, then $\phi$ is reducible.
\end{lemma}
\begin{proof}
    If $p=0$, then $\det(\phi)=1$ gives that $rq=1$. Since $r,q\in \Z$, this forces that $r=q=\pm 1$. Then the trace of $\phi$ is $\pm 2$ and thus $\phi$ is reducible.
\end{proof}

\section{Main results}

\subsection{Genericity of $\Q$-homology spheres}
\begin{proposition}\cite[Proposition 11.3.1]{Mar16}\label{Prop: 1stZHom}
    Let $L\subset S^3$ be a link with $k$ components and $M$ its complement. We have $H_1(M,\Z)=\Z^k$, generated by the $k$ meridians.
\end{proposition}

Recall that $\xi: \sl(2,\Z)\to \Q\cup\{\infty\}$ is defined by $\xi(\begin{pmatrix}
    r & s\\ p & q
\end{pmatrix})=p/q$. 

\begin{lemma}
    Let $K\subset S^3$ be a knot and $M$ its complement. If $\phi\in \sl(2,\Z)$ is not reducible, then the $\xi(\phi)$-Dehn filling of $M$ along $\partial M$ gives a $\Q$-homology sphere.
\end{lemma}
\begin{proof}
    Let $\phi\in \sl(2,\Z)$ be a non-reducible element and $\bar M$ be the $\xi(\phi)$-Dehn filling of $M$ along $\partial M$. 

     Suppose that $\phi=\begin{pmatrix}
    r & s\\ p & q
\end{pmatrix}$. By Lemma \ref{Lem: p=0}, $p\neq 0$. Fix a meridian $m$ and a longitude $l$ of $\partial M$. Then $\xi(\phi)$-Dehn filling kills the curve $pm+ql$. Since $p\neq 0$, together with Proposition \ref{Prop: 1stZHom}, we get that $H_1(\bar M, \Z)=\Z/p\Z$. Then the conclusion follows.
\end{proof}

Similarly, we have
\begin{lemma}\label{Lem: AnosovImpQHomSph}
    Let $L\subset S^3$ be a link with $k$ components and $M$ its complement. If $\phi_1,\cdots,\phi_k\in \sl(2,\Z)$ are not reducible, then the $(\xi(\phi_1),\cdots,\xi(\phi_k))$-Dehn filling of $M$ along $\partial M$ gives a $\Q$-homology sphere.
\end{lemma}

By combining Corollary \ref{Cor: AnoIsGen} and Lemma \ref{Lem: AnosovImpQHomSph}, we have

\begin{proposition}\label{Prop: GenQHomSph}
    Let $L\subset S^3$ be a link with $k$ components and $M$ its complement. A generic Dehn surgery on $L$ gives a $\Q$-homology sphere, i.e. $$\lim_{R\to \infty}\frac{\splitfrac{\sharp\{\phi_1,\cdots,\phi_k\in B_F(R):}{ (\xi(\phi_1),\cdots,\xi(\phi_k))\text{-Dehn filling of } M \text{ along } \partial M \text{ gives a } \Q\text{-homology sphere}\}}}{\sharp B_F(R)^k}= 1$$ for any finite generating set $F$ of $\sl(2,\Z)$.
\end{proposition}

Let $B_n$ be the braid group with $n>3$ strands. Each element $f$ in $B_n$ determines a link in $S^3$ which is the closure $\hat f$ of $f$. Conversely, a well-known theorem of Alexander \cite{Ale23} says that every oriented link diagram may be isotoped to a closed braid. Hence, it is natural to take $B_n$ as a model to study the random phenomenon of knots and links. 

As a corollary of Proposition \ref{Prop: GenQHomSph}, one gets that

\begin{corollary}\label{Cor: GenQHomSph}
    For every $n>3$, a generic Dehn surgery on the closure of a braid in $B_n$ gives a $\Q$-homology sphere, i.e. 
    $$\lim_{R\to \infty}\sum_{k=1}^n\frac{\splitfrac{\sharp\{f\in B_F(R),\phi_1,\cdots,\phi_k\in B_{F'}(R):\hat f \text{ is a link with }k \text{ components, }}{ (\xi(\phi_1),\cdots,\xi(\phi_k))\text{-Dehn filling of } S^3\setminus N(\hat f) \text{ gives a } \Q\text{-homology sphere}\}}}{\sharp B_F(R)\sharp B_{F'}(R)^k}= 1$$
    for any finite generating set $F$ of $B_n$ and any finite generating set $F'$ of $\sl(2,\Z)$.
\end{corollary}
\begin{proof}
    For any braid $f\in B_n$, it is clear that the closure $\hat f$ of $f$ has at most $n$ components. Hence, for any $R>0$ and any finite generating set $F$ of $B_n$, we have $$\sum_{k=1}^n\frac{\sharp\{f\in B_F(R): \hat f \text{ is a link with }k \text{ components}\}}{\sharp B_F(R)}=1.$$

    Fix a link $\hat f$ with $k$ components and a finite generating set $F'$ of $\sl(2,\Z)$. By Proposition \ref{Prop: GenQHomSph}, for any $\epsilon>0$, there exists $R_k>0$ such that $$\frac{\splitfrac{\sharp\{\phi_1,\cdots,\phi_k\in B_{F'}(R):}{ (\xi(\phi_1),\cdots,\xi(\phi_k))\text{-Dehn filling of } S^3\setminus N(\hat f) \text{ gives a } \Q\text{-homology sphere}\}}}{\sharp B_{F'}(R)^k}>1-\epsilon$$
    for any $R\ge R_k$. By letting $R\ge\max\{R_1,\cdots, R_n\}$, one has that $$\sum_{k=1}^n\frac{\splitfrac{\sharp\{f\in B_F(R),\phi_1,\cdots,\phi_k\in B_{F'}(R):\hat f \text{ is a link with }k \text{ components, }}{ (\xi(\phi_1),\cdots,\xi(\phi_k))\text{-Dehn filling of } S^3\setminus N(\hat f) \text{ gives a } \Q\text{-homology sphere}\}}}{\sharp B_F(R)\sharp B_{F'}(R)^k}>1-\epsilon.$$ Letting $\epsilon\to 0$ completes the proof.
\end{proof}

Together with Theorem \ref{Thm: LW} and Alexander's Theorem \cite{Ale23}, we have
\begin{theorem}\label{Thm: GenQHomSph}
    A generic closed orientable 3-manifold is a $\Q$-homology sphere: for any finite generating set $F_n$ of $B_n(n\ge 3)$ and any finite generating set $F$ of $\sl(2,\Z)$, there exists an unbounded and monotonically increasing function $\alpha: \N_{\ge 3}\to \N_{\ge 3}$ such that 
    $$\lim_{n\to \infty}\frac{1}{\sharp B_{F_n}(n)}\sum_{k=1}^n\frac{\splitfrac{\sharp\{f\in B_{F_n}(n),\phi_1,\cdots,\phi_k\in B_{F}(\alpha(n)):\hat f \text{ is a link with }k \text{ components, }}{ (\xi(\phi_1),\cdots,\xi(\phi_k))\text{-Dehn filling of } S^3\setminus N(\hat f) \text{ gives a } \Q\text{-homology sphere}\}}}{\sharp B_{F}(\alpha(n))^k}= 1.$$
\end{theorem}
\begin{proof}
    For every $n\ge 3$, fix a finite generating set $F_n$ of $B_n$ and a finite generating set $F$ of $\sl(2,\Z)$. By choosing $\epsilon=\frac{1}{n}$, the proof of Corollary \ref{Cor: GenQHomSph} shows that there exists $\alpha(n)\in \N$ such that $$\frac{1}{\sharp B_{F_n}(n)}\sum_{k=1}^n\frac{\splitfrac{\sharp\{f\in B_{F_n}(n),\phi_1,\cdots,\phi_k\in B_{F}(\alpha(n)):\hat f \text{ is a link with }k \text{ components, }}{ (\xi(\phi_1),\cdots,\xi(\phi_k))\text{-Dehn filling of } S^3\setminus N(\hat f) \text{ gives a } \Q\text{-homology sphere}\}}}{\sharp B_{F}(\alpha(n))^k}>1-\frac{1}{n}.$$
    By letting $n\to \infty$ and requiring $\alpha(n)\ge \max\{\alpha(n-1),n\}$, we complete the proof.
\end{proof}

\subsection{Genericity of non-$\Z$-homology spheres}

Next, we are going to show that a generic closed orientable 3-manifold is not a $\Z$-homology sphere.

At first, it easily follows from Proposition \ref{Prop: 1stZHom} that
\begin{lemma}\label{Lem: p=1}
    Let $L\subset S^3$ be a link with $k$ components and $M$ its complement. Let $\phi_i=\begin{pmatrix}
    r_i & s_i\\ p_i & q_i
\end{pmatrix}\in \sl(2,\Z)$ for $1\le i\le k$. Then the $(\xi(\phi_1),\cdots,\xi(\phi_k))$-Dehn filling of $M$ along $\partial M$ gives a $\Z$-homology sphere if and only if $p_i=\pm 1$ for each $1\le i\le k$.
\end{lemma}

\begin{fact}\label{Fac2}
    \begin{enumerate}
        \item\cite[\textsection 3.1.1]{FM12} $\sl(2,\Z)$ is generated by 
        $$\begin{pmatrix}
    1 & 0\\ 1 & 1
\end{pmatrix} \text{ and } 
\begin{pmatrix}
    1 & 1\\ 0 & 1
\end{pmatrix}.$$
    \item\cite[Example 4.4.1, Proposition 4.4.2]{Loh17} Let $G\le \sl(2,\Z)$ be a subgroup generated by 
    $$\begin{pmatrix}
    1 & 0\\ 2 & 1
\end{pmatrix} \text{ and } 
\begin{pmatrix}
    1 & 2\\ 0 & 1
\end{pmatrix}.$$
    Then $G\cong \F_2$ and $[\sl(2,\Z):G]=12$. 
    \end{enumerate}
\end{fact}

The above fact shows that $\sl(2,\Z)$ is a 2-generated virtually free group. In this subsection, we fix a standard generating set $F$ given by Fact \ref{Fac2} (1) of $\sl(2,\Z)$. Denote $|\cdot |_F$ as the word norm on $\sl(2,\Z)$ with respect to $F$. 

\begin{lemma}\label{Lem: NormEstimate}
    There exist two constants $a_1,a_2>0$ such that for any $r,q\in \Z$, we have $$|\begin{pmatrix}
    r+1 & rq+r+q\\ 1 & q+1
\end{pmatrix}|_F\ge a_1(|r|+|q|)-a_2.$$
\end{lemma}
\begin{proof}
    Denote $S=\{\begin{pmatrix}
    1 & 0\\ 2 & 1
\end{pmatrix},
\begin{pmatrix}
    1 & 2\\ 0 & 1
\end{pmatrix}\}\subset B_F(2)$ and $G=\langle S\rangle\cong \F_2$ by Fact \ref{Fac2} (2). Let $\iota: G\to \sl(2,\Z)$ be the natural embedding. Since $G$ is a finite-index subgroup of $\sl(2,\Z)$, $\iota$ induces a quasi-isometry between two Cayley graphs $\G(G,S)$ and $\G(\sl(2,\Z),F)$. This shows that there exists a constant $a\ge 1$ such that $$\frac{1}{a}|A|_{S}-a\le |A|_F\le 2|A|_{S}$$ for any $A\in G$.

    Now, let us consider the matrix $A=\begin{pmatrix}
    r+1 & rq+r+q\\ 1 & q+1
\end{pmatrix}\in \sl(2,\Z)$. It is straightforward to compute that $A=\begin{pmatrix}
    1 & 1\\ 0 & 1
\end{pmatrix}^r\begin{pmatrix}
    1 & 0\\ 1 & 1
\end{pmatrix}\begin{pmatrix}
    1 & 1\\ 0 & 1
\end{pmatrix}^q=A_1A_2$ where $A_1=\begin{pmatrix}
    1 & 1\\ 0 & 1
\end{pmatrix}^r$ and $A_2=\begin{pmatrix}
    1 & 0\\ 1 & 1
\end{pmatrix}\begin{pmatrix}
    1 & 1\\ 0 & 1
\end{pmatrix}^q=\begin{pmatrix}
    1 & q\\ 1 & q+1
\end{pmatrix}$. Next, we are going to find a matrix $A'\in G$ such that $|A'|_F$ is approximately equal to $|A|_F$.

    Note that $A_1$ can be reformulated as $\begin{pmatrix}
    1 & 2\\ 0 & 1
\end{pmatrix}^{\frac{r}{2}}$ when $r$ is even or $\begin{pmatrix}
    1 & 1\\ 0 & 1
\end{pmatrix}\begin{pmatrix}
    1 & 2\\ 0 & 1
\end{pmatrix}^{\frac{r-1}{2}}$ when $r$ is odd. Denote $$B=\begin{pmatrix}
    1 & 2\\ 0 & 1
\end{pmatrix}\begin{pmatrix}
    1 & 0\\ -2 & 1
\end{pmatrix}\begin{pmatrix}
    1 & 2\\ 0 & 1
\end{pmatrix}\begin{pmatrix}
    1 & 0\\ -2 & 1
\end{pmatrix}=\begin{pmatrix}
    5 & -4\\ 4 & -3
\end{pmatrix}\in B_{S}(4)$$ and $q=4m+i$ for some $m\in \Z, i\in \{0,1,2,3\}$. An important observation is that $A_2=\begin{pmatrix}
    1 & q\\ 1 & q+1
\end{pmatrix}$, $BA_2=\begin{pmatrix}
    1 & q-4\\ 1 & q-3
\end{pmatrix}$ and then $$B^mA_2=\begin{pmatrix}
    1 & q-4m\\ 1 & q-4m+1
\end{pmatrix}=\begin{pmatrix}
    1 & i\\ 1 & i+1
\end{pmatrix}=\begin{pmatrix}
    1 & 0\\ 1 & 1
\end{pmatrix}\begin{pmatrix}
    1 & 1\\ 0 & 1
\end{pmatrix}^i\in B_F(i+1)\subseteq B_F(4).$$ The above two reformulations, together with the triangle inequality, show that there exists a matrix $A'\in G$ such that $|A|_F\ge |A'|_F-5$. Indeed, $A'$ can be explicitly expressed as $\begin{pmatrix}
    1 & 2\\ 0 & 1
\end{pmatrix}^{\frac{r}{2}}B^{-m}$ when $r$ is even or $\begin{pmatrix}
    1 & 2\\ 0 & 1
\end{pmatrix}^{\frac{r-1}{2}}B^{-m}$ when $r$ is odd. Since $S$ is a free basis of $G$, it is easy to see that $|A'|_{S}\ge \frac{|r|}{2}+|q|-6$. Therefore, one gets that $$|A|_F\ge |A'|_F-5\ge \frac{1}{a}|A'|_{S}-a-5\ge \frac{1}{a}(\frac{|r|}{2}+|q|-6)-a-5.$$
Letting $a_1:=\frac{1}{2a}$ and $a_2=5+a+\frac{6}{a}$ completes the proof.

\end{proof}

Denote $$\z=\left\{\begin{pmatrix}
    r+1 & rq+r+q\\ 1 & q+1
\end{pmatrix}:r,q\in \Z\right\}\subset \sl(2,\Z).$$

\begin{lemma}\label{Lem: Neg}
    $\z$ is negligible in $\sl(2,\Z)$, i.e. 
    \begin{equation}\label{Equ1}
        \lim_{R\to \infty}\frac{\sharp(B_{F'}(R)\cap \z)}{\sharp B_{F'}(R)}= 0
    \end{equation}
    for any finite generating set $F'$ of $\sl(2,\Z)$.
\end{lemma}
\begin{proof}
    We first show that Formula (\ref{Equ1}) holds for the standard generating set $F$. By Lemma \ref{Lem: NormEstimate}, for a matrix $A=\begin{pmatrix}
    r+1 & rq+r+q\\ 1 & q+1
\end{pmatrix}\in B_F(R)\cap \z$, there are at most $\frac{2(R+a_2)}{a_1}+1$ choices for each of $r$ and $q$. Hence, $\sharp(B_F(R)\cap \z)\le (\frac{2(R+a_2)}{a_1}+1)^2$.

As shown in the proof of Lemma \ref{Lem: NormEstimate}, $\sl(2,\Z)$ has a free subgroup $G$ of index 12 and $B_F(2)$ contains a free basis of $G$. Hence, $\sharp B_F(R)\ge 3^{\lfloor R/2\rfloor}$.

Therefore, we have $$\frac{\sharp(B_{F}(R)\cap \z)}{\sharp B_{F}(R)}\le \frac{(\frac{2(R+a_2)}{a_1}+1)^2}{3^{\lfloor R/2\rfloor}}\to 0$$ as $R\to \infty$.

Now, we assume that $F'$ is an arbitrary finite generating set of $\sl(2,\Z)$. Since $[\sl(2,\Z): G]=12$, the well-known Shalen-Wagreich's Proposition (cf. \cite[Proposition 3.3]{SW92}) shows that $B_{F'}(23)$ contains a free basis of $G$. Hence, $\sharp B_{F'}(R)\ge 3^{\lfloor R/23\rfloor}$. Let $k:=\max_{A\in F'}|A|_F$. Then $B_{F'}(R)\subseteq B_F(kR)$. This shows that $\sharp(B_{F'}(R)\cap \z)\le \sharp(B_{F}(kR)\cap \z)\le (\frac{2(kR+a_2)}{a_1}+1)^2$. Therefore, we have $$\frac{\sharp(B_{F'}(R)\cap \z)}{\sharp B_{F'}(R)}\le \frac{(\frac{2(kR+a_2)}{a_1}+1)^2}{3^{\lfloor R/23\rfloor}}\to 0$$ as $R\to \infty$.
\end{proof}

By combining Lemma \ref{Lem: p=1} and Lemma \ref{Lem: Neg}, we have

\begin{proposition}\label{Prop: GenNonZHomSph}
    Let $L\subset S^3$ be a link with $k$ components and $M$ its complement. A generic Dehn surgery on $L$ gives a non-$\Z$-homology sphere, i.e. $$\lim_{R\to \infty}\frac{\splitfrac{\sharp\{\phi_1,\cdots,\phi_k\in B_{F'}(R):}{ (\xi(\phi_1),\cdots,\xi(\phi_k))\text{-Dehn filling of } M \text{ along } \partial M \text{ gives a } \Z\text{-homology sphere}\}}}{\sharp B_{F'}(R)^k}= 0$$ for any finite generating set $F'$ of $\sl(2,\Z)$. 
\end{proposition}

Similar to Theorem \ref{Thm: GenQHomSph}, one has that
\begin{theorem}\label{Thm: GenNonZHomSph}
    A generic closed orientable 3-manifold is not a $\Z$-homology sphere: for any finite generating set $F_n$ of $B_n(n\ge 3)$ and any finite generating set $F$ of $\sl(2,\Z)$, there exists an unbounded and monotonically increasing function $\alpha: \N_{\ge 3}\to \N_{\ge 3}$ such that 
    $$\lim_{n\to \infty}\frac{1}{\sharp B_{F_n}(n)}\sum_{k=1}^n\frac{\splitfrac{\sharp\{f\in B_{F_n}(n),\phi_1,\cdots,\phi_k\in B_{F}(\alpha(n)):\hat f \text{ is a link with }k \text{ components, }}{ (\xi(\phi_1),\cdots,\xi(\phi_k))\text{-Dehn filling of } S^3\setminus N(\hat f) \text{ gives a non-} \Z\text{-homology sphere}\}}}{\sharp B_{F}(\alpha(n))^k}= 1.$$
\end{theorem}

\subsection{Genericity of hyperbolic links}

Recall that $B_n$ is the braid group with $n>3$ strands and each element $f$ in $B_n$ determines a link in $S^3$ which is the closure $\hat f$ of $f$. In \cite{Ma14}, Ma has showed that for a random walk $w_{n,k}$ of length $k$ on $B_n$, the probability that the closure $\widehat{w_{n,k}}$ is a hyperbolic link in $S^3$ converges to 1 as $k\to \infty$. 
For the sake of consistency, we use the counting measure on $B_n$ to re-declare that a generic braid in $B_n$ gives a hyperbolic link.

Let $D$ be a disk with $n>3$ punctures. The braid group $B_n$ can be seen as the mapping class group of $D$. Let $\mathscr C(D)$ denote the curve graph of $D$ and $d_{\mathscr C(D)}$ denote the combinatorial metric on $\mathscr C(D)$. As a result of \cite[Theorem 1.3]{Bow08}, $B_n$ admits an acylindrical action on $\mathscr C(D)$. For each element $f\in B_n$ and a basepoint $x\in \mathscr C(D)$, we denote $$\tau(f):=\lim_{m\to \infty}\frac{d_{\mathscr C(D)}(x,f^m(x))}{m}$$ as the \textit{stable translation length} of $f$ on $\mathscr C(D)$. Note that $\tau(f)$ does not depend on the choice of $x\in \mathscr C(D)$. The following result is a special case of \cite[Theorem 1.1]{Cho25b}.

\begin{lemma}\label{Lem: pA in groups}
    For any finite generating set $F$ of $B_n$ and any $C>0$, we have
    $$\lim_{R\to \infty}\frac{\sharp\{f\in B_F(R): f \text{ is pseudo-Anosov and satisfies } \tau(f)>C\}}{\sharp B_F(R)}= 1.$$
\end{lemma}

Fix a sufficiently large $C>0$. Let $f\in B_n$ be a pseudo-Anosov element with $\tau(f)>C$. By the hyperbolization theorem for fibered 3-manifolds, the mapping torus $M_f=D\times [0,1]/\{(z,0)\sim(f(z),1)\}$ is hyperbolic. Note that $S^3\setminus N(\hat f)$ can be seen as a Dehn filling of $M_f$ along ONE cusp. Let $\lambda$ denote the killed curve in $\partial M_f$. Since $\tau(f)$ is sufficiently large, the formula (2.3) in \cite{Ma14} shows that the normalized length of $\lambda$ is also sufficiently large. Then by \cite[Theorem 5.11]{HK08}, $S^3\setminus N(\hat f)$ is hyperbolic. In summary, we have

\begin{theorem}\label{Thm: HypLink}
    A generic braid in $B_n$ gives a hyperbolic link, i.e. 
    $$\lim_{R\to \infty}\frac{\sharp\{f\in B_F(R): \hat f \text{ is a hyperbolic link}\}}{\sharp B_F(R)}= 1$$
    for any finite generating set $F$ of $B_n$.
\end{theorem}

\subsection{Genericity of hyperbolic knots}
There is a natural map $\pi$ from $B_n$ to $S_n$, the permutation group on $n$ letters $\Omega = \{1, 2,\ldots,n\}$: let $\sigma_1,\ldots,\sigma_{n-1}$ be the canonical generators of $B_n$--that is,
when viewing $B_n$ as the mapping class group of a disk with $n$ punctures, $\sigma_i$ is the (right) half-twist which permutes the $i$-th and $(i + 1)$-th punctures--and let $s_i$ be the simple transposition on the letters $i$ and $i+ 1$ for $1 \le i \le n-1$; then $\pi(\sigma_i)=s_i$. The number of components of our link $\hat f$ is just the number of orbits of $\pi(f)$ on $\Omega = \{1, 2,\ldots,n\}$, i.e., the number of a presentation of $\pi(f)$ as a product of disjoint cycles. The probability that $\hat f$ is a knot is exactly the probability that $\pi(f)$ is an $n$-cycle. Denote $$\C_n=\{f\in B_n: \pi(f) \text{ is an } n-\text{cycle}\}.$$

To show the positive density of $\C_n$ in $B_n$, we need the following lemma which can be extracted from \cite[Corollary 1]{CW18}.
\begin{lemma}\cite[Lemma 3.12]{CW25}\label{Lem: CW}
    Let $G$ be a group with a finite generating set $F$. Let $E\subset G$ be a subset. Suppose there is a finite subset $B\subset G$ with the following property: For any $g\in G$, there exists $b\in B$ such that $gb\in E$. Then there exists a constant $c=c(F)\in (0,1)$ such that $$\frac{\sharp  (B_F(R)\cap E)}{\sharp B_F(R)}>c$$ for any sufficiently large $R$.
\end{lemma}

Let $sec: S_n\to B_n$ be a section of the quotient map $\pi$. Denote $B=sec(S_n)$.

\begin{lemma}\label{Lem: Cn}
    For any $g\in B_n$, there exists $b\in B$ such that $gb\in \C_n$.
\end{lemma}
\begin{proof}
    Let $a\in S_n$ such that $\pi(g)a$ is an $n$-cycle. Then $b=sec(a)$ satisfies the requirement.
\end{proof}

Applying Lemma \ref{Lem: Cn} to Lemma \ref{Lem: CW} gives

\begin{lemma}\label{Lem: PosDen}
    $\C_n$ has a positive density in $B_n$. That is, for every finite generating set $F$ of $B_n$, there exists a constant $c=c(F)\in (0,1)$ such that
    $$\frac{\sharp (B_F(R)\cap \C_n)}{\sharp B_F(R)}>c$$
    for any sufficiently large $R$.
\end{lemma}

\begin{remark}
    We remark that $\C_n$ is not generic in $B_n$, i.e. the constant $c$ in Lemma \ref{Lem: PosDen} can not be arbitrarily close to 1. The reason is simple: if one considers the density of the set $\C_{n-1}=\{f\in B_n: \pi(f) \text{ is an } (n-1)-\text{cycle}\}$ in $B_n$, then the same arguments as above show that $\C_{n-1}$ has a positive density in $B_n$. This in turn gives the non-genericity of $\C_n$.
\end{remark}

\begin{lemma}\label{Lem: pA in subsets}
    For any finite generating set $F$ of $B_n$ and any $C>0$, we have
    $$\lim_{R\to \infty}\frac{\sharp\{f\in B_F(R)\cap \C_n: f \text{ is pseudo-Anosov and satisfies } \tau(f)>C\}}{\sharp (B_F(R)\cap\C_n)}= 1.$$
\end{lemma}
\begin{proof}
    Suppose to the contrary that there exist a finite generating set $F$ of $B_n$ and two constants $C>0, c_1\in (0,1)$ such that $$\frac{\sharp\{f\in B_F(R)\cap \C_n: f \text{ is pseudo-Anosov and satisfies } \tau(f)>C\}}{\sharp (B_F(R)\cap\C_n)}\le c_1$$ for any $R\gg 0$. Let $c_2=c_2(F)$ be given by Lemma \ref{Lem: PosDen}. Denote $c_3=1-c_2+c_1c_2\in (0,1)$.

    By Lemma \ref{Lem: pA in groups}, for any sufficiently large $R$, one has 
    \begin{equation}\label{Equ: c3}
        \frac{\sharp\{f\in B_F(R): f \text{ is pseudo-Anosov and satisfies } \tau(f)>C\}}{\sharp B_F(R)}>c_3.
    \end{equation}
    
    Fix a sufficiently large $R$ so that all the above inequalities hold. Then,
    \begin{align*}
        \frac{\sharp\{f\in B_F(R): f \text{ is } \ldots\}}{\sharp B_F(R)}&=\frac{\sharp\{f\in B_F(R)\cap \C_n: f \text{ is }\ldots\}+\sharp\{f\in B_F(R)\setminus \C_n: f \text{ is }\ldots\}}{\sharp (B_F(R)\cap \C_n)+\sharp (B_F(R)\setminus \C_n)}\\
        &\le \frac{\sharp\{f\in B_F(R)\cap \C_n: f \text{ is }\ldots\}+\sharp (B_F(R)\setminus \C_n)}{\sharp (B_F(R)\cap \C_n)+\sharp (B_F(R)\setminus \C_n)}\\
        &\le \frac{c_1\sharp(B_F(R)\cap \C_n)+\sharp (B_F(R)\setminus \C_n)}{\sharp (B_F(R)\cap \C_n)+\sharp (B_F(R)\setminus \C_n)}\\
        &=\frac{c_1\sharp B_F(R)+(1-c_1)\sharp (B_F(R)\setminus \C_n)}{\sharp B_F(R)}\\
        &\le c_1+(1-c_1)(1-c_2)=c_3.
    \end{align*}
    This is a contradiction to Formula (\ref{Equ: c3}). Then the conclusion follows.
\end{proof}

As a consequence of Lemma \ref{Lem: pA in subsets}, we have
\begin{theorem}\label{Thm: HypKnot}
    A generic braid in $\C_n$ gives a hyperbolic knot, i.e. 
    $$\lim_{R\to \infty}\frac{\sharp\{f\in B_F(R)\cap \C_n: \hat f \text{ is a hyperbolic knot}\}}{\sharp (B_F(R)\cap \C_n)}=1$$
    for any finite generating set $F$ of $B_n$.
\end{theorem}


\subsection{Genericity of hyperbolic 3-manifolds}

Let $L\subset S^3$ be a hyperbolic link with $k$ components. A fundamental theorem of Thurston states that, for all but a finite number of Dehn surgeries on $L$, the resulting closed 3-manifold has a hyperbolic structure. Those that don’t result in a hyperbolic structure are called \textit{exceptional} curves. In \cite{HK05}, Hodgson-Kerckhoff showed that there is a universal bound to the number of curves that must be excluded from each boundary component; in particular, when $k=1$, there are at most 60 exceptional curves and when $k>1$, at most 114 curves from each component need to be excluded. 

For each pair of coprime integers $(p,q)$, there exists a pair of coprime integers $(r,s)$ such that $rq-sp=1$. If there exists another pair of integers $(r',s')$ satisfying that $r'q-s'p=1$, then the two pairs differ by an integer multiple of $(p,q)$.

Fix a pair of coprime integers $(p,q)$. Let $(r,s)$ be a pair of integers such that $rq-sp=1$. 

Let $F$ be the standard generating set of $\sl(2,\Z)$ given by Fact \ref{Fac2} (1). 

\begin{lemma}\label{Lem: NormEstimate2}
    For any $m\in \Z$, we have $$|m|-|\begin{pmatrix}
    r & s\\ p & q
\end{pmatrix}|_F\le |\begin{pmatrix}
    r+mp & s+mq\\ p & q
\end{pmatrix}|_F\le |\begin{pmatrix}
    r & s\\ p & q
\end{pmatrix}|_F+|m|.$$
\end{lemma}
\begin{proof}
    Note that $$\begin{pmatrix}
    r+mp & s+mq\\ p & q
\end{pmatrix}=\begin{pmatrix}
    1 & m\\ 0 & 1
\end{pmatrix}\begin{pmatrix}
    r & s\\ p & q
\end{pmatrix}=\begin{pmatrix}
    1 & 1\\ 0 & 1
\end{pmatrix}^m\begin{pmatrix}
    r & s\\ p & q
\end{pmatrix}. $$

By the triangle inequality of $|\cdot|_F$, one has
$$\left||\begin{pmatrix}
    r+mp & s+mq\\ p & q
\end{pmatrix}|_F-|\begin{pmatrix}
    1 & 1\\ 0 & 1
\end{pmatrix}^m|_F\right|=\left||\begin{pmatrix}
    r+mp & s+mq\\ p & q
\end{pmatrix}|_F-|m|\right|\le|\begin{pmatrix}
    r & s\\ p & q
\end{pmatrix}|_F.$$
\end{proof}

Denote $$\E_{p,q}=\left\{\begin{pmatrix}
    r+mp & s+mq\\ p & q
\end{pmatrix}:m\in \Z\right\}\subset \sl(2,\Z).$$

\begin{lemma}\label{Lem: Neg2}
    $\E_{p,q}$ is negligible on $\sl(2,\Z)$, i.e. 
    \begin{equation}\label{Equ2}
        \lim_{R\to \infty}\frac{\sharp(B_{F'}(R)\cap \E_{p,q})}{\sharp B_{F'}(R)}= 0
    \end{equation}
    for any finite generating set $F'$ of $\sl(2,\Z)$.
\end{lemma}
\begin{proof}
    We follow the same proof line of Lemma \ref{Lem: Neg}. First, we show that Formula (\ref{Equ2}) holds for the standard generating set $F$. Denote $k=|\begin{pmatrix}
    r & s\\ p & q
\end{pmatrix}|_F$. By Lemma \ref{Lem: NormEstimate2}, for a matrix $A=\begin{pmatrix}
    r+mp & s+mq\\ p & q
\end{pmatrix}\in B_F(R)\cap \E_{p,q}$, there are at most $2(R+k)+1$ choices for $m$. Hence, $\sharp(B_F(R)\cap \E_{p,q})\le 2R+2k+1$. Since $\sharp B_F(R)$ grows exponentially with respect to $R$, we obtain Formula (\ref{Equ2}) for the standard generating set $F$. The remaining proofs are completed by the same arguments in the proof of Lemma \ref{Lem: Neg}.
\end{proof}

Let $L\subset S^3$ be a hyperbolic link with $k$ components. To obtain a closed hyperbolic 3-manifold via Dehn surgeries on $L$, Hodgson-Kerckhoff \cite{HK05} showed that there are at most 114 exceptional curves from each component. Let $M=S^3\setminus N(L)$ be the complement of $L$ in $S^3$ with $k$ torus boundaries $T_1,\cdots,T_k$. Let $A_i=\{(p_{i1},q_{i1}),\cdots,(p_{i114},q_{i114})\}$ be the set of exceptional curves from $T_i$. Let $A_L=A_1\cup \cdots A_k$. Since $A_L$ is a finite subset, it follows from Lemma \ref{Lem: Neg2} that
\begin{lemma}\label{Lem: UniNeg}
    $\bigcup_{(p,q)\in A_L}\E_{p,q}$ is negligible on $\sl(2,\Z)$.
\end{lemma}

Hence, we obtain that
\begin{theorem}\label{Thm: DehnSurImpHyp3mfd}
    Let $L\subset S^3$ be a hyperbolic link with $k$ components. A generic Dehn surgery on $L$ gives a closed hyperbolic 3-manifold, i.e. $$\lim_{R\to \infty}\frac{\sharp\{\phi_1,\cdots,\phi_k\in B_{F'}(R): \xi(\phi_1),\cdots,\xi(\phi_k)\text{ are non-exceptional slopes on }L \}}{\sharp B_{F'}(R)^k}= 1$$ for any finite generating set $F'$ of $\sl(2,\Z)$. 
\end{theorem}
\begin{proof}
    If $\xi(\phi)$ is an exceptional slope on $L$ for some $\phi\in \sl(2,\Z)$, then $\phi\in \bigcup_{(p,q)\in A_L}\E_{p,q}$. Then the conclusion follows from Lemma \ref{Lem: UniNeg}.
\end{proof}

By combining Theorems \ref{Thm: LW}, \ref{Thm: HypLink}, \ref{Thm: DehnSurImpHyp3mfd} and Alexander's Theorem \cite{Ale23}, we have
\begin{theorem}\label{Thm: GenHyp3Mfd}
    A generic closed orientable 3-manifold is hyperbolic: for any finite generating set $F_n$ of $B_n$ and any finite generating set $F$ of $\sl(2,\Z)$, there exist two unbounded and monotonically increasing functions $\alpha, \beta: \N_3\to \N_3$ such that 
    $$\lim_{n\to \infty}\frac{1}{\sharp B_{F_n}(\alpha(n))}\sum_{k=1}^n\frac{\splitfrac{\sharp\{f\in B_{F_n}(\alpha(n)),\phi_1,\cdots,\phi_k\in B_{F}(\beta(n)):\hat f \text{ is a hyperbolic link with }}{k \text{ components, }\xi(\phi_1),\cdots,\xi(\phi_k)\text{ are non-exceptional slopes on }\hat f \}}}{\sharp B_{F}(\beta(n))^k}= 1.$$
\end{theorem}
\begin{proof}
    For every $n\ge 3$, fix a finite generating set $F_n$ of $B_n$ and a finite generating set $F$ of $\sl(2,\Z)$. Theorems \ref{Thm: HypLink} shows that there exists $\alpha(n)\in \N$ such that 
    \begin{align*}
        &\frac{\sharp\{f\in B_{F_n}(\alpha(n)): \hat f \text{ is a hyperbolic link}\}}{\sharp B_{F_n}(\alpha(n))}\\ &=\sum_{k=1}^n\frac{\sharp\{f\in B_{F_n}(\alpha(n)): \hat f \text{ is a hyperbolic link with }k \text{ components}\}}{\sharp B_{F_n}(\alpha(n))}>1-\frac{1}{n}.
    \end{align*}

    Since $B_{F_n}(\alpha(n))$ is a finite set, Theorems \ref{Thm: DehnSurImpHyp3mfd} shows that there exists $\beta(n)\in \N$ such that 
    $$\frac{\sharp\{\phi_1,\cdots,\phi_k\in B_{F}(\beta(n)): \xi(\phi_1),\cdots,\xi(\phi_k)\text{ are non-exceptional slopes on }\hat f \}}{\sharp B_{F}(\beta(n))^k}> 1-\frac{1}{n}$$ for every hyperbolic link $\hat f$ with $k$ components obtained from the closure of $f\in B_{F_n}(\alpha(n))$.
    
    Therefore, one has that $$\frac{1}{\sharp B_{F_n}(\alpha(n))}\sum_{k=1}^n\frac{\splitfrac{\sharp\{f\in B_{F_n}(\alpha(n)),\phi_1,\cdots,\phi_k\in B_{F}(\beta(n)):\hat f \text{ is a hyperbolic link with }}{k \text{ components, }\xi(\phi_1),\cdots,\xi(\phi_k)\text{ are non-exceptional slopes on }\hat f \}}}{\sharp B_{F}(\beta(n))^k}>(1-\frac{1}{n})^2.$$
    By letting $n\to \infty$ and requiring $\alpha(n)\ge \max\{\alpha(n-1),n\}, \beta(n)\ge \max\{\beta(n-1),n\}$, we complete the proof.
\end{proof}

\begin{remark}
    Although a generic closed orientable 3-manifold is hyperbolic (see Theorem \ref{Thm: GenHyp3Mfd} or \cite[Theorem 1.10]{HYZ25}) and has non-vanishing first $\Z$-homology group ((see Theorem \ref{Thm: GenNonZHomSph} or \cite[Corollary 8.5]{DT06})), Lubotzki, Maher and Wu \cite{LMW16} constructed an infinite family of closed 3-manifolds which are hyperbolic $\Z$-homology spheres. 
\end{remark}

\bibliographystyle{amsplain}   
\bibliography{Reference}
\end{document}